\documentclass{amsart}
\usepackage{ifthen,amsfonts,amsmath,amssymb,epic,eepic}
\usepackage{epsfig,color}
\usepackage{verbatim}

\makeatother

\newtheorem{Thm}{Theorem}
\newtheorem{Cor}{Corollary}
\newtheorem{Lem}{Lemma}

\numberwithin{equation}{section}
\newcommand{\dd}{\displaystyle}
\newcommand{\nn}{{\bf{n}}}
\newcommand{\K}{{\text{K}}}

\def\RR{{\bold R}}

\def\CC{{\bold C }}

\newcommand{\e}{{\text {e}}}

\newcommand{\eqr}[1]{(\ref{#1})}

\begin{document}

\title[Minimal Lamination with Singularities on a Line Segment]
{A Minimal Lamination of the Unit Ball with Singularities along a Line Segment}

\author{Siddique Khan}%
\address{Department of Mathematics\\
Johns Hopkins University\\
3400 N. Charles St.\\
Baltimore, MD 21218}
\thanks{The author would like to thank Bill Minicozzi, Christine Breiner and Hamid Hezari for many helpful discussions which were instrumental in developing the example in this paper.}
\email{siddique@alum.mit.edu}

\begin{abstract}
We construct a sequence of compact embedded minimal disks in the unit ball in Euclidean 3-space whose boundaries are in the boundary of the ball and where the curvatures blow up at every
point of a line segment of the vertical axis, extending from the origin. We further study the
transversal structure of the minimal limit lamination and find removable singularities along the line segment and a non-removable singularity at the origin. This
extends a result of Colding and Minicozzi where they constructed a
sequence with curvatures blowing up only at the center of the
ball, Dean's construction of a sequence with curvatures
blowing up at a prescribed discrete set of points, and the classical case of
the sequence of re-scaled helicoids with curvatures blowing up
along the entire vertical axis.
\end{abstract}

\maketitle

\section{Introduction}
In the study and classification of Minimal Surfaces an important question is what are the
possible singular sets for limits of sequences of embedded minimal surfaces.
The global problem in $\RR^3$ is understood by results of Colding
and Minicozzi \cite{CM2}-\cite{CM5}, where all singularities are
removable at least when the sequence is simply connected, and by the work of Meeks and Rosenberg in \cite{MR} where they explain why the singular set is perpendicular to the limit foliation.

In contrast, for the local case in $\RR^3$, Colding and Minicozzi in
\cite{CM1} prove the existence of a sequence of embedded minimal
disks with boundaries in a sphere and with curvatures blowing up
only at the center of the ball, where there is a non-removable
singularity. A result of Dean, \cite{BD}, extends this example by
constructing a sequence with curvatures blowing up at a prescribed
discrete set of points. There is also the well known case of
the sequence of re-scaled helicoids with curvature blowing up along the entire $x_3$-axis.
Meeks and Weber \cite{MW} construct singular sets that are properly embedded $C^{1,1}$-curves.
Hoffman and White \cite{HW}, and Calle and Lee \cite{CL} also give a variational way to compute examples.

In this paper we construct a sequence of compact embedded minimal
disks in the unit ball in $\RR^3$ whose boundaries are in
the boundary of the ball and where the curvatures blow up at every
point of a line segment of the negative $x_3$-axis. This sequence converges to a minimal limit lamination.
We study the transversal structure of the limit lamination and find a foliation by parallel planes in the lower hemisphere and a leaf in the upper hemisphere that spirals into the $\{x_3=0\}$ plane, such that there are removable singularities at every point along the line segment of the negative $x_3$-axis but the singularity at the origin cannot be removed.

We will follow the structure of Colding and Minicozzi's result in \cite{CM1}. The key difference in our
approach here is that we alter the domain and the Weierstrass data
used in \cite{CM1} to create not just one singularity converging
to the origin, but a sequence of singularities that converges to a
line segment extending from the origin. In addition, the construction of the limit lamination uses a convergence result from \cite{CM5}. And a Bernstein-type theorem is used to obtain the foliation by parallel planes in the lower hemisphere of the unit ball.

Our main result is Theorem \ref{t:main} below, which constructs our sequence of compact embedded minimal disks in the unit ball with boundaries in the boundary of the ball and describes the limit lamination. Theorem \ref{t:main1} first constructs a sequence of compact embedded minimal disks with the necessary curvature and boundary properties.

\begin{Thm} \label{t:main1}
There exists  a sequence of compact embedded minimal disks $0 \in
M_N \subset \RR^2 \times [-1/2,1/2] \subset \RR^3$, each containing the vertical segment $\{ (0,0,t) \, | \, |t| \leq 1/2 \}
\subset M_N$, with the following properties:
\begin{enumerate}
\item[(a)] \label{i:1} $\displaystyle \forall p \in \{ (0,0,t) \,
| \, -1/2 \leq t \leq 0 \}, \, \lim_{N\to \infty} |A_{M_N}|^2
(p) = \infty$.

\item[(b)] \label{i:3} $M_N \setminus \{(0,0,t) |\, |t| \leq 1/2 \} = {M}_{1,N} \cup {M}_{2,N}$
 for multi-valued graphs ${M}_{1,N} , \, {M}_{2,N}$ over the $\{x_3=0\} \setminus \{0\}$ punctured plane.

\item[(c)] \label{i:2} $\sup_N \sup_{M_N \setminus B_{\delta}
} |A_{M_N}|^2 = C_{\delta} < \infty$ for all $\delta > 0$ and some constant $C_{\delta}$ depending on $\delta$, and where $B_{\delta}$ is a ${\delta}$-neighborhood of $\{ (0,0,t) \, | \, -1/2\leq t \leq 0 \}$ .

\item[(d)] The boundary $\partial M_N$ lies outside a fixed cylinder $\{ (x_1, x_2, x_3) |\, x_1^2+x_2^2 \leq r_0^2, -1/2 < x_3 < 1/2 \}$ where $r_0$ does not depend on $N$. Also, in each horizontal slice $\{x_3=t\} \cap M_N,$ for  $-1/2 \leq t \leq 0$ (i.e. below the $\{x_3=0\}$ plane) the distance from the $x_3$-axis to $\partial M_N$ goes to infinity as $N \to \infty$.
\end{enumerate}
\end{Thm}

Figure \ref{multipleslices} shows a diagram of horizontal slices of ${M}_{1,N} , \, {M}_{2,N}$.
\begin{center}
\begin{figure}[ht]
 \includegraphics[width=\textwidth]{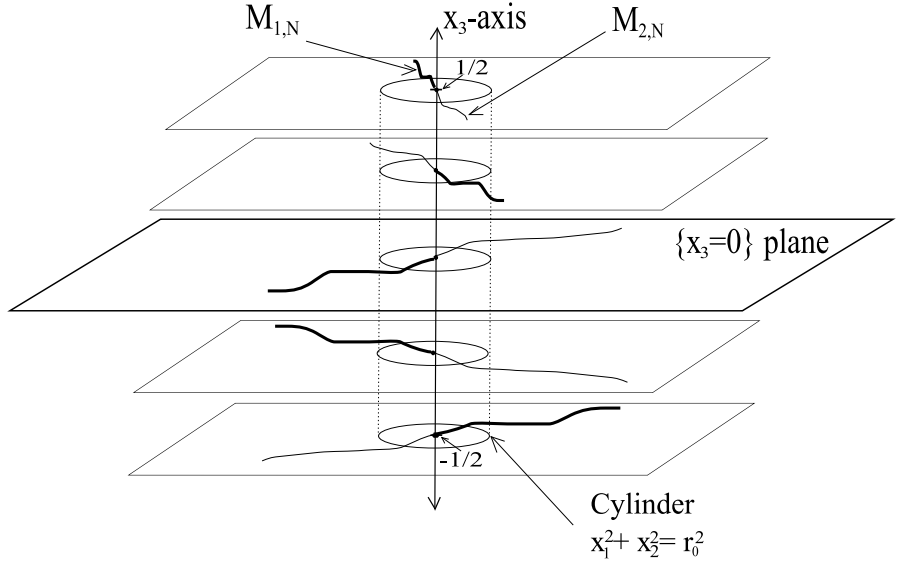}
 \caption{Horizontal slices of $M_N\setminus \{(0,0,t) |\, |t| \leq 1/2 \} = {M}_{1,N} \cup {M}_{2,N}$ in Corollary \ref{c:corembeddings}}\label{multipleslices}
\end{figure}
\end{center}

The boundary properties of the sequence in Theorem \ref{t:main1}(d) allow us to intersect with a smaller ball that is contained in the cylinder, pass to a subsequence, and then scale to obtain the sequence in Theorem \ref{t:main} below that has the same properties (a), (b) and (c). Theorem \ref{t:main} further describes the convergence of this sequence to a limit lamination of the unit ball with singularities along the line segment $\{ (0,0,t) \, | \, -1\leq t \leq 0 \}$.

\begin{Thm} \label{t:main}
There exists  a sequence  of compact embedded minimal disks $0 \in
\Sigma_N \subset  B_{1} \subset \RR^3$
 with $\partial \Sigma_N \subset
\partial B_{1}$ and each
containing the vertical segment $\{ (0,0,t) \, | \, |t| \leq 1 \}
\subset \Sigma_N$ with the following properties:
\begin{enumerate}
\item[(a)] \label{i:21} $\displaystyle \forall p \in \{ (0,0,t) \,
| \, -1\leq t \leq 0 \}, \, \lim_{N\to \infty} |A_{\Sigma_N}|^2
(p) = \infty$.

\item[(b)] \label{i:23} $\Sigma_N \setminus \{(0,0,t) |\, |t| \leq 1 \} = {\Sigma}_{1,N} \cup {\Sigma}_{2,N}$
 for multi-valued graphs ${\Sigma}_{1,N} , \, {\Sigma}_{2,N}$ over the $\{x_3=0\} \setminus \{0\}$ punctured plane.

\item[(c)] \label{i:22} $\sup_N \sup_{\Sigma_N \setminus B_{\delta}
} |A_{\Sigma_N}|^2 < \infty$ for all $\delta > 0$, where
$B_{\delta}$ is a ${\delta}$-neighborhood of $\{ (0,0,t) \, | \, -1\leq t \leq 0 \}$ .

\end{enumerate}

\noindent This sequence of compact embedded minimal disks converges to a minimal lamination of $B_{1}
\setminus \{ (0,0,t) \, | \, -1\leq t \leq 0 \}$ consisting of a foliation by parallel planes of the lower hemisphere below $\{x_3=0\}$ and one leaf in the upper hemisphere, $\Sigma$, such that $\Sigma \setminus \{x_3-\text{axis}\}= \Sigma' \cup \Sigma''$, where $\Sigma'$ and $\Sigma''$ are multi-valued graphs, each of which spirals into $\{x_3=0\}$. This limit lamination has removable singularities along the line segment $\{ (0,0,t) \, | \, -1\leq t \leq 0 \}$ of the negative $x_3$-axis but the singularity at the origin cannot be removed.
\end{Thm}

Figure \ref{myexample} shows a schematic picture of this limit lamination.

\begin{center}
\begin{figure}[ht]
 \includegraphics[width=0.75\textwidth]{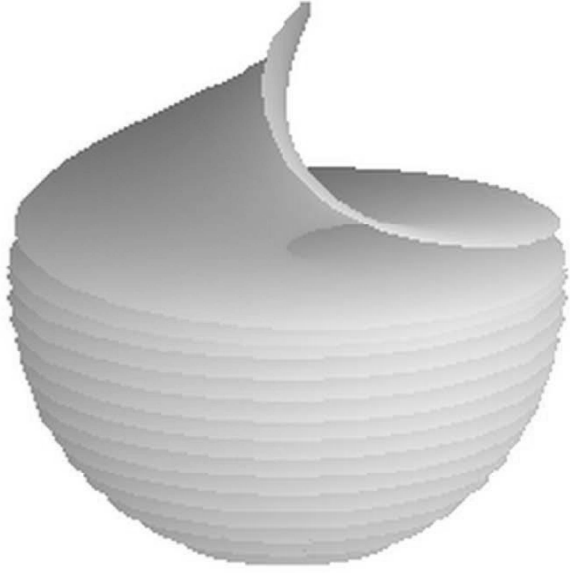}
 \caption{Schematic picture of the limit lamination in Theorem \ref{t:main}}\label{myexample}
\end{figure}
\end{center}

\section{Notation}

In this paper we will use standard $(x_1 , x_2 , x_3)$ coordinates
on $\RR^3$ and $z=x+i\, y$ on $\CC$.  Given a function $f:\CC \to
\CC^n$, we will use $\partial_x f$ and $\partial_y f$ to denote
$\frac{\partial f}{\partial x}$ and $\frac{\partial f}{\partial
y}$, respectively. Also, $\partial_z f =  (\partial_x f - i
\partial_y f) /2$.

For $p \in \RR^3$ and $s> 0$, the ball of radius $s$ in $\RR^3$ will be denoted by
$B_s(p)$.
$\K_{\Sigma}$ is the sectional curvature of a smooth
surface $\Sigma$. When $\Sigma$ is immersed in $\RR^3$, we let its second fundamental form be $A_{\Sigma}$ (hence when $\Sigma$ is minimal, $|A_{\Sigma}|^2=-2\,\K_{\Sigma}$). When $\Sigma$
is oriented, we let $\nn_{\Sigma}$ be the unit normal.

\section{The Weierstrass representation}

Given a meromorphic function $g(z)$ and a holomorphic one-form
$\phi(z)$, defined on a domain $\Omega$,
the Weierstrass representation of a conformal minimal immersion, $F:\Omega \rightarrow \mathbb{R}^3$, is given by (see \cite{Os}, Lemma 8.2):\\
\begin{equation} \label{weierstrassrep} F(z) = \text{Re} \int_{\zeta \in
\gamma_{z_0,z}} \left(\frac{1}{2}(g^{-1}(\zeta)-g(\zeta)),
\frac{i}{2}(g^{-1}(\zeta)+g(\zeta)),1 \right) \phi(\zeta)
\end{equation},
where $z_0$ is a fixed base point in $\Omega$ and the integration
is taken along a path $\gamma_{z_0,z}$ from $z_0$ to $z$ in
$\Omega$. In this paper we set the base point $z_0=0$, we choose
$\phi$ such that it has no zeros and $g$ such that it has no poles
or zeros in $\Omega$ and we choose $\Omega$ to be simply connected. These ensure that $F(z)$ does not depend on the choice of path $\gamma_{z_0,z}$ and that the differential of $F$, $dF$, is non-zero (and this ensures that $F$ is an immersion).

The unit normal $\textbf{n}$ and the Gauss curvature $\textbf{K}$
of this surface are (see \cite{Os} sections 8, 9):
\begin{equation} \label{n} \textbf{n}=(2\text{Re} g, 2 \text{Im} g,
|g|^2-1)/(|g|^2+1) \end{equation} \begin{equation} \label{K}
\textbf{K}=-\left[\frac{4|\partial_z g|
|g|}{|\phi|(1+|g|^2)^2}\right]^2 \end{equation}

We will need the following lemma, which follows immediately from the Weierstrass representation \eqr{weierstrassrep}, that gives the differential of $F$:
\begin{Lem} \label{l:dF}
If $F$ is given by \eqr{weierstrassrep} with $g (z) = \e^{i \,
(u(z)+i v(z))}$ and $\phi= dz$, then
\begin{align}
 \partial_x F
    &=  ( \sinh v  \, \cos u ,  \sinh v  \, \sin u  , 1)  \label{e:dxF} \,
    , \\
    \partial_y F
    &= (\cosh v \,  \sin u , - \cosh v \,  \cos u , 0)  \, . \label{e:dyF}
\end{align}
\end{Lem}

\section{Proof of the Main Theorems}

We proceed to prove Theorem \ref{t:main1} by first constructing a family of minimal immersions $F_N$ with a specific choice of Weierstrass data $g(z)=e^{i h_N(z)} , \phi(z)=dz$, where $h_N(z)=u_N + i v_N$ and a corresponding domain $\Omega_N$ to obtain $F_N(z)$ from \eqr{weierstrassrep}.

 We first define $\displaystyle \partial_{z}h_N(z)$ because it is this derivative that will be essential in determining the curvature properties required of our sequence of embedded minimal disks. We let:

  \[\displaystyle \partial_{z}h_N(z) = \frac{1}{2}\left[\frac{1}{[z^2+(\frac{1}{N})^2]^2}+\frac{1}{N}\sum_{k=1}^{N} \frac{1}{[(z+\frac{k}{N})^2+(\frac{1}{N})^2]^2} \right],\]

for $N \geq 2$ on the domain $\Omega_N = \Omega_N^+ \cup \Omega_N^- \, ,$ where:
\[\displaystyle \Omega_N^+ =\left\{(x,y) \, \left|  \, |y| \leq
\frac{(x^2+(\frac{1}{N})^2)^{5/4}}{4},\, 0 < x \leq 1/2 \right. \right\} \,\]
and $$ \displaystyle \, \Omega_N^- = \left\{(x,y) \, \left| \, |y| \leq
b_N, -1/2 \leq x \leq 0
\right. \right\},$$ where $b_N=\dfrac{1}{4N^{5/2}}$. See Figure \ref{domain}.

In this paper, we will denote the upper and lower boundary of
$\Omega_N$ as $y_{x,N}$ and $-y_{x,N}$ respectively. That is, on $\Omega_N^+$, we set
$y_{x,N}=\frac{(x^2+(\frac{1}{N})^2)^{5/4}}{4}$ and on
$\Omega_N^-$, we set $y_{x,N}=b_N$.

We note that for all $N$,  $\partial_{z}h_N(z)$ is holomorphic on the domain $\Omega_N$ because the poles $\{\pm \frac{i}{N}-\frac{k}{N}\}$ for $0 \leq k \leq N$ lie outside the domain. Furthermore, these poles converge to the line segment $\{-1 \leq x \leq 0 \}$ as $N \to \infty$.

\begin{figure}[tbp]
    \setlength{\captionindent}{4pt}
    \begin{minipage}[t]{0.5\textwidth}
    \centering\includegraphics[width=\textwidth]{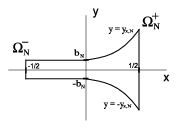}
    \caption{Diagram of the domain, $\Omega_N$}\label{domain}
    \end{minipage}\begin{minipage}[t]{0.5\textwidth}
    \centering\includegraphics[width=\textwidth]{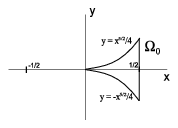}
    \caption{Diagram of the limit domain, $\Omega_0$}\label{omega0}
    \end{minipage}
\end{figure}

Lemma \ref{l:uniformconv} shows that there is a subsequence, $\{\partial_z h_{N_i}(z) \}$, that converges uniformly to a limit that we denote as: $$\displaystyle
\partial_z h(z) = \lim_{N_i \rightarrow \infty} \partial_z h_N(z)$$  on compact subsets of $$\displaystyle
\Omega_0 = \bigcap_N \Omega_N \setminus \{0\} = \left\{(x,y) \, \left| \, |y| \leq \frac{x^{5/2}}{4}, 0 < x
\leq 1/2
\right. \right\}$$
(See Figure \ref{omega0}). Therefore since for each $N$, $\partial_z h_N(z)$ is holomorphic on $\Omega_N$, we have that $\partial_z h(z),$ and hence also $h(z)$, is holomorphic on $\Omega_0$.

\begin{Lem} \label{l:uniformconv}
For $N \to \infty$, there is a subsequence $\{N_i\}$ such that
$\{\partial_z h_{N_i}(z) \}$ converges uniformly on compact subsets of
$\Omega_0$.
\end{Lem}

\begin{proof}
For every compact subset, $K \subset \Omega_0 , \, \exists \,
r_K>0$ such that $\forall z \in K , \, \forall N,k>0 , \, |(z+\frac{k}{N})^2+(\frac{1}{N})^2|>r_K , \, |z^2+(\frac{1}{N})^2|>r_K \\$
$\displaystyle \implies  \left|\partial_{z}h_N(z)\right| =
\left|
\frac{1}{2}\left[\frac{1}{[z^2+(\frac{1}{N})^2]^2}+\frac{1}{N}\sum_{k=1}^{N}
\frac{1}{[(z+\frac{k}{N})^2+(\frac{1}{N})^2]^2} \right] \right| \\ <
\frac{1}{2}\left[\frac{1}{r_K^2}+\frac{1}{r_K^2}\right] < \frac{1}{r_K^2}$

$\implies \{ \partial_z h_N \}$ is a family of holomorphic
functions , bounded on compact subsets of $\Omega_0$

$\implies$ by Montel's theorem, there is a subsequence that
converges uniformly on compact subsets of $\Omega_0$ to a
holomorphic limit.
\end{proof}

Now, since for each $N$, $\partial_z h_N(z)$ is holomorphic on $\Omega_N$, we integrate to obtain our Weierstrass data: \begin{eqnarray} \lefteqn{ \label{hNz} \displaystyle h_N(z) =
\frac{N^2}{4}\left[N\arctan(Nz)+\frac{z}{z^2+(\frac{1}{N})^2} \right.} \\ & &
+ \left.\frac{1}{N}\sum_{k=1}^{N}
\left(N\arctan(N(z+k/N))+\frac{(z+k/N)}{(z+k/N)^2+(\frac{1}{N})^2}\right)\right] \nonumber
\end{eqnarray}

And since $h_N(z)$ is also holomorphic on $\Omega_N,$ by the
Cauchy-Riemann equations we have:
\[\partial_{z}h_N(z) = \partial_{x}u_N- i\partial_{y}u_N = \partial_{y}v_N+ i\partial_{x}v_N \]

Therefore:

\begin{eqnarray}
  \lefteqn{\displaystyle
\partial_{z}h_N(z) = \frac{1}{2}\left[\frac{1}{[z^2+(\frac{1}{N})^2]^2}+ \frac{1}{N}\sum_{k=1}^{N}
\frac{1}{[(z+\frac{k}{N})^2+(\frac{1}{N})^2]^2}\right]} \nonumber \\
& & = \frac{1}{2}\left[\frac{(x^2 +(\frac{1}{N})^2 - y^2)^2 - 4x^2y^2 -
4ixy(x^2 +(\frac{1}{N})^2 - y^2)}{([x^2 +(\frac{1}{N})^2 - y^2]^2+4x^2y^2)^2} \right. \nonumber \\
& & + \frac{1}{N}\sum_{k=1}^{N} \left( \frac{((x+k/N)^2 +(\frac{1}{N})^2 -
y^2)^2 - 4(x+k/N)^2y^2}{([(x+k/N)^2 +(\frac{1}{N})^2 - y^2]^2+4(x+k/N)^2y^2)^2} \right. \nonumber \\
& & \left. \left. - \frac{4i(x+k/N)y((x+k/N)^2 +(\frac{1}{N})^2 -
y^2)}{([(x+k/N)^2 +(\frac{1}{N})^2 - y^2]^2+4(x+k/N)^2y^2)^2}\right) \right] \nonumber
\end{eqnarray}

\begin{eqnarray}
\lefteqn{ \implies \partial_{y}u_N = \frac{1}{2}\left[\frac{4xy(x^2
+(\frac{1}{N})^2 - y^2)}{([x^2 +(\frac{1}{N})^2 - y^2]^2+4x^2y^2)^2}\right.} \label{dyuN} \\
&& + \left. \frac{1}{N}\sum_{k=1}^{N} \frac{ 4(x+k/N)y((x+k/N)^2 +(\frac{1}{N})^2 -
y^2)}{([(x+k/N)^2 +(\frac{1}{N})^2 - y^2]^2+4(x+k/N)^2y^2)^2}\right] \nonumber
\end{eqnarray}

And \begin{eqnarray} \lefteqn{\displaystyle
\partial_{y}v_N = \frac{1}{2}\left[\frac{(x^2 +(\frac{1}{N})^2 - y^2)^2 - 4x^2y^2}{([x^2 +(\frac{1}{N})^2 - y^2]^2+4x^2y^2)^2} \right.} \label{dyvN} \\
&& + \left. \frac{1}{N}\sum_{k=1}^{N} \frac{((x+k/N)^2 +(\frac{1}{N})^2 - y^2)^2 -
4(x+k/N)^2y^2}{([(x+k/N)^2 +(\frac{1}{N})^2 -
y^2]^2+4(x+k/N)^2y^2)^2}\right] \nonumber
\end{eqnarray}

Now the main difficulty we encounter in the proof of Theorem \ref{t:main1} is showing that the immersions
$F_N : \Omega_N \to \RR^3$ are in fact embeddings.

The next Lemma gives this embeddedness result.

\begin{Lem}
\label{lem3}
There exists $r_0>0$ (independent of $N$) such that $\forall (x,y)
\in \Omega_N,$ \label{embeddedness}
\begin{eqnarray} && x_3 ( F_N (x,y) ) = x  \, . \label{e:zequalx} \\
&& \text{The curve } F_N(x,\cdot) : [-y_{x,N} , y_{x,N}] \to \{ x_3 = x \} \label{e:graphical} \\
&& \text{ is a graph in the $\{x_3=x\}$ plane} \, . \nonumber\\
&&  |F_N(x , \pm y_{x,N}) - F_N(x,0)| > r_0\text{ for all } N \, . \label{e:boundary}\\
&& \text{In fact, for} \,\, x \leq 0, \, |F_N(x , \pm y_{x,N}) - F_N(x,0)| \to \infty \, \text{as} \, N \to \infty \label{e:boundaryinfinity}
\end{eqnarray}
\end{Lem}

In this Lemma, \eqr{e:zequalx} shows that the horizontal slice of the image, $F_N(\Omega_N) \cap \{x_3=t\}$, is the image of the vertical line $\{x=t\}$ in the domain $(\Omega_N)$. \eqr{e:graphical} shows that the image $F_N(\{x=t\} \cap \Omega_N)$ is a graph in the $\{x_3=t\}$ plane over a line segment in that plane (see Figure \ref{horizontalslice}). Together, these imply embeddedness. Also, \eqr{e:boundary} shows that there is some $r_0$ such that the boundary of the graph in \eqr{e:graphical} lies outside a circle $B_{r_0}(F_N(t,0))$ for all $N$. And \eqr{e:boundaryinfinity} shows that for all $x \leq 0$ (i.e. in the part of the image $F_N(\Omega_N)$ below the $\{x_3=0\}$ plane), these boundaries of the graph in \eqr{e:graphical} actually go to infinity as $N \to \infty$.

\begin{center}
\begin{figure}[ht]
 \includegraphics[width=\textwidth]{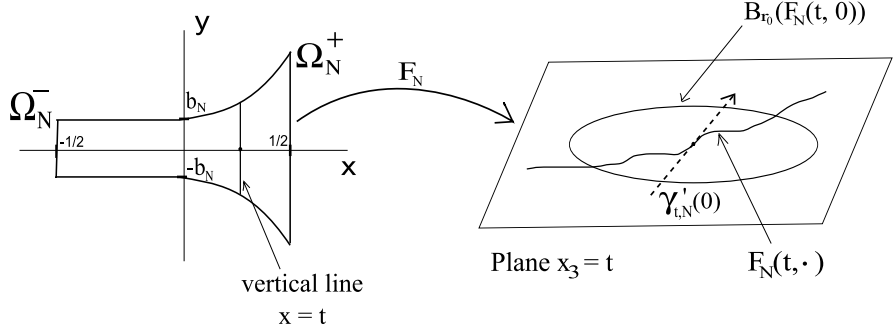}
 \caption{A horizontal slice of $F(\Omega_N)$ in Lemma \ref{embeddedness}}\label{horizontalslice}
\end{figure}
\end{center}

\begin{proof}
Since $z_0=0$ and the height differential is $\phi=dz,\,$
\eqr{e:zequalx} follows immediately from \eqr{weierstrassrep}.

Now we prove \eqr{e:graphical} first for $0<x \leq \frac{1}{2}$
(i.e. on $\Omega_N^+$) and then for $-\frac{1}{2} \leq x \leq 0$
(i.e. on $\Omega_N^-$)

\subsection{\textbf{Proof of \eqr{e:graphical} on $\Omega_N^{+}$}}
We first note that on
\[\Omega_N^{+} =\left\{(x,y) \, \left|  \, |y| \leq \frac{(x^2+(\frac{1}{N})^2)^{5/4}}{4},\, 0 < x \leq 1/2 \right.
\right\}, \]
\begin{equation}
\label{y} \displaystyle
 y^2 \leq \frac{(x^2+(\frac{1}{N})^2)^{5/2}}{16}
\leq \frac{x^2+(\frac{1}{N})^2}{16} \leq \frac{(x+
\frac{k}{N})^2+(\frac{1}{N})^2}{16};
\end{equation} \begin{equation} y^2 \leq \frac{(x^2+(\frac{1}{N})^2)^{5/2}}{16} \leq \frac{((x+ \frac{k}{N})^2+(\frac{1}{N})^2)^{5/2}}{16}\end{equation}
for all $k \geq 1 \,$ since $0<x \leq 1/2$ and $N
\geq 2 \implies (\frac{1}{N}) \leq 1/2$

Now using this in \eqr{dyuN},

\begin{eqnarray}
\lefteqn{\displaystyle |\partial_{y}u_N| \leq \frac{1}{2}\left[\frac{4x|y||x^2 +(\frac{1}{N})^2 - y^2|}{([x^2 +(\frac{1}{N})^2 - y^2]^2+4x^2y^2)^2} \right.} \nonumber \\
&& \left. + \frac{1}{N}\sum_{k=1}^{N} \frac{ 4(x+k/N)|y||(x+k/N)^2 +(\frac{1}{N})^2 -
y^2|}{([(x+k/N)^2 +(\frac{1}{N})^2 - y^2]^2+4(x+k/N)^2y^2)^2}\right] \nonumber \\
&& \leq \frac{1}{2}\left[ \frac{4x|y|(x^2 +(\frac{1}{N})^2)}{([\frac{15}{16}(x^2 +(\frac{1}{N})^2)]^2)^2} \right. \nonumber\\
&& \left. + \frac{1}{N}\sum_{k=1}^{N} \frac{ 4(x+k/N)|y|((x+k/N)^2 +(\frac{1}{N})^2)}{([\frac{15}{16}((x+k/N)^2 +(\frac{1}{N})^2)]^2)^2}\right] \nonumber\\
&& = 2\left(\frac{16}{15}\right)^4|y|\left[\frac{x}{(x^2 +(\frac{1}{N})^2)^3}+\frac{1}{N}\sum_{k=1}^{N}
\frac{(x+\frac{k}{N})}{((x+\frac{k}{N})^2 +(\frac{1}{N})^2)^3}\right] \nonumber \end{eqnarray}

We set $\displaystyle y_{x,N}=\frac{(x^2+(\frac{1}{N})^2)^{5/4}}{4}\\
\implies \max_{|y| \leq y_{x,N}} |u_N(x,y) - u_N(x,0)| \\
\leq 2\left(\frac{16}{15}\right)^4 \left(\int_0^{y_{x,N}} t \, dt
\right)\left[\frac{x}{(x^2 +(\frac{1}{N})^2)^3}+\frac{1}{N}\sum_{k=1}^{N}
\frac{(x+\frac{k}{N})}{((x+\frac{k}{N})^2 +(\frac{1}{N})^2)^3} \right] \\ =
\left(\frac{16}{15}\right)^4 {y_{x,N}}^2 \left[\frac{x}{(x^2
+(\frac{1}{N})^2)^3}+\frac{1}{N}\sum_{k=1}^{N}
\frac{(x+\frac{k}{N})}{((x+\frac{k}{N})^2 +(\frac{1}{N})^2)^3} \right]  \\
= \left(\frac{16}{15}\right)^4 \frac{(x^2+(\frac{1}{N})^2)^{5/2}}{16}
\left[\frac{x}{(x^2 +(\frac{1}{N})^2)^3}+\frac{1}{N}\sum_{k=1}^{N}
\frac{(x+\frac{k}{N})}{((x+\frac{k}{N})^2 +(\frac{1}{N})^2)^3} \right]  \\
= \left(\frac{16}{15}\right)^4 \frac{1}{16}
\left[\frac{x(x^2+(\frac{1}{N})^2)^{5/2}}{(x^2
+(\frac{1}{N})^2)^3}+\frac{1}{N}\sum_{k=1}^{N}
\frac{(x+\frac{k}{N})((x+\frac{k}{N})^2+(\frac{1}{N})^2)^{5/2}}{((x+\frac{k}{N})^2
+(\frac{1}{N})^2)^3} \right] \\
 =  \left(\frac{16}{15}\right)^4 \frac{1}{16}
\left[\frac{x}{(x^2 +(\frac{1}{N})^2)^{\frac{1}{2}}}+\frac{1}{N}\sum_{k=1}^{N}
\frac{(x+\frac{k}{N})}{((x+\frac{k}{N})^2 +(\frac{1}{N})^2)^{\frac{1}{2}}}
\right]$
\begin{equation} \leq  \left(\frac{16}{15}\right)^4 \frac{1}{16}[1+1] < 1 \label{e:maxu} \end{equation}

We set $\gamma_{x,N}(y)=F_N(x,y) \implies \gamma'_{x,N}(y)=\partial_y F_N(x,y)$. We note that $v_N(x,0)=0$ from \eqr{hNz} and that $\cos(1)>1/2$.

Therefore \eqr{e:dyF} gives: \begin{equation} \label{e:angle}
\langle \gamma'_{x,N}(y),\gamma'_{x,N}(0) \rangle = \cosh v_N(x,y)
\cos(u_N(x,y)-u_N(x,0)) > \frac{1}{2} \cosh v_N(x,y)
\end{equation}

Hence, by \eqr{e:angle}, the angle between $\gamma'_{x,N}(y)$ and
$\gamma'_{x,N}(0)$ is always less than $\pi/2$. This gives us
\eqr{e:graphical} for all $0<x \leq \frac{1}{2}$ (i.e. on
$\Omega_N^+$)

Furthermore, this result holds uniformly in $N$ because of the
uniform bound in \eqr{e:maxu}.

\subsection{\textbf{Proof of \eqr{e:graphical} on $\Omega_N^{-}$}}
\vskip1cm
 We recall that $\Omega_N^{-} = \{(x,y) \, | \, |y| \leq
b_N, -1/2 \leq x \leq 0 \}$, where $b_N=\dfrac{1}{4N^{5/2}}$

Now we note that on $\Omega_N^{-}$, for $N \geq 2,$
\begin{equation} \label{boundary4}
y^2 \leq b_N^2=\dfrac{1}{16N^{5}} \leq \frac{(x^2+(\frac{1}{N})^2)^{5/2}}{16}\end{equation}
and \begin{equation} \label{boundary3}
y^2 \leq b_N^2=\dfrac{1}{16N^{5}} \leq \frac{((x+k/N)^2+(\frac{1}{N})^2)^{5/2}}{16}\end{equation}
Also, \begin{equation} \label{boundary2}
y^2 \leq b_N^2=\dfrac{1}{16N^{5}}<\frac{x^2+(\frac{1}{N})^2}{16}\end{equation}
and \begin{equation} \label{boundary1}
y^2 \leq b_N^2=\dfrac{1}{16N^{5}}<\frac{(x+k/N)^2+(\frac{1}{N})^2}{16}\end{equation}

Now using these inequalities in \eqr{dyuN},

$\displaystyle |\partial_{y}u_N| \leq \frac{1}{2}\left[\frac{4x|y||x^2
+(\frac{1}{N})^2 - y^2|}{([x^2 +(\frac{1}{N})^2 - y^2]^2+4x^2y^2)^2} \right. \\
+ \left. \frac{1}{N}\sum_{k=1}^{N} \frac{ 4(x+k/N)|y||(x+k/N)^2 +(\frac{1}{N})^2 -
y^2|}{([(x+k/N)^2 +(\frac{1}{N})^2 - y^2]^2+4(x+k/N)^2y^2)^2}\right] \\ \leq
\frac{1}{2}\left[\frac{4x|y|(x^2
+(\frac{1}{N})^2)}{([\frac{15}{16}(x^2 +(\frac{1}{N})^2)]^2)^2} \\
+ \frac{1}{N}\sum_{k=1}^{N} \frac{ 4(x+k/N)|y|((x+k/N)^2
+(\frac{1}{N})^2)}{([\frac{15}{16}((x+k/N)^2 +(\frac{1}{N})^2)]^2)^2}\right]
\\ = 2\left(\frac{16}{15}\right)^4|y|\left[\frac{x}{(x^2
+(\frac{1}{N})^2)^3}+\frac{1}{N}\sum_{k=1}^{N}
\frac{(x+\frac{k}{N})}{((x+\frac{k}{N})^2 +(\frac{1}{N})^2)^3}
\right] $

We recall that $\displaystyle y_{x,N}=b_N \\
\implies \max_{|y| \leq y_{x,N}} |u_N(x,y) - u_N(x,0)| \\
\leq 2\left(\frac{16}{15}\right)^4 \left(\int_0^{y_{x,N}} t \, dt
\right)\left[\frac{x}{(x^2 +(\frac{1}{N})^2)^3}+\frac{1}{N}\sum_{k=1}^{N}
\frac{(x+\frac{k}{N})}{((x+\frac{k}{N})^2 +(\frac{1}{N})^2)^3} \right] \\ =
\left(\frac{16}{15}\right)^4 {y_{x,N}}^2 \left[\frac{x}{(x^2
+(\frac{1}{N})^2)^3}+\frac{1}{N}\sum_{k=1}^{N}
\frac{(x+\frac{k}{N})}{((x+\frac{k}{N})^2 +(\frac{1}{N})^2)^3} \right]  \\
= \left(\frac{16}{15}\right)^4 b_N^2 \left[\frac{x}{(x^2
+(\frac{1}{N})^2)^3}+\frac{1}{N}\sum_{k=1}^{N}
\frac{(x+\frac{k}{N})}{((x+\frac{k}{N})^2 +(\frac{1}{N})^2)^3} \right]  \\
\leq \left(\frac{16}{15}\right)^4 \frac{1}{16}
\left[\frac{x(x^2+(\frac{1}{N})^2)^{5/2}}{(x^2
+(\frac{1}{N})^2)^3}+\frac{1}{N}\sum_{k=1}^{N}
\frac{(x+\frac{k}{N})((x+\frac{k}{N})^2+(\frac{1}{N})^2)^{5/2}}{((x+\frac{k}{N})^2
+(\frac{1}{N})^2)^3} \right] \\  \displaystyle = \left(\frac{16}{15}\right)^4 \frac{1}{16}
\left[\frac{x}{(x^2 +(\frac{1}{N})^2)^{\frac{1}{2}}}+\frac{1}{N}\sum_{k=1}^{N}
\frac{(x+\frac{k}{N})}{((x+\frac{k}{N})^2 +(\frac{1}{N})^2)^{\frac{1}{2}}}
\right]$
\begin{equation} \label{e:maxu2} \leq \left(\frac{16}{15}\right)^4 \frac{1}{16}[1+1] < 1 \end{equation}

And now we use the same argument we used to show \eqr{e:graphical}
on $\Omega_N^+$.
We set $\gamma_{x,N}(y)=F_N(x,y)$. We note that
$v_N(x,0)=0$ and $\cos(1)>1/2$.

Therefore \eqr{e:dyF} gives:
\begin{eqnarray} && \langle \gamma'_{x,N}(y),\gamma'_{x,N}(0) \rangle = \cosh v_N(x,y) \cos(u_N(x,y)-u_N(x,0)) \label{e:angle2}\\
 && > \frac{1}{2} \cosh v_N(x,y) \nonumber
\end{eqnarray}
Hence, by \eqr{e:angle2}, the angle between $\gamma'_{x,N}(y)$ and
$\gamma'_{x,N}(0)$ is always less than $\pi/2$. This gives us
\eqr{e:graphical} for all $-1/2 \leq x \leq 0$ (i.e. on
$\Omega_N^-$)

Furthermore, this result holds uniformly in $N$ because of the
uniform bound in \eqr{e:maxu2}.
Now we prove \eqr{e:boundary} first for $0<x \leq \frac{1}{2}$ (i.e. on $\Omega_N^+$) and then for $-\frac{1}{2} \leq x \leq 0$ (i.e. on $\Omega_N^-$).
\subsection{\textbf{Proof of \eqr{e:boundary} on $\Omega_N^{+}$}}
We recall that $\\ \Omega_N^{+}=\left\{(x,y) \, \left|  \,
|y| \leq \frac{(x^2+(\frac{1}{N})^2)^{5/4}}{4},\, 0 < x \leq 1/2 \right.
\right\},$

From \eqr{dyvN}, and by using \eqr{y} we have:
\begin{eqnarray*}
  \lefteqn{\displaystyle \partial_{y}v_N = \frac{1}{2}\left[\frac{(x^2 +(\frac{1}{N})^2 - y^2)^2 - 4x^2y^2}{([x^2 +(\frac{1}{N})^2 - y^2]^2+4x^2y^2)^2} \right.}\\
  & & + \left. \frac{1}{N}\sum_{k=1}^{N} \frac{((x+\frac{k}{N})^2 +(\frac{1}{N})^2 - y^2)^2 -
4(x+\frac{k}{N})^2y^2}{([(x+\frac{k}{N})^2 +(\frac{1}{N})^2 - y^2]^2+4(x+\frac{k}{N})^2y^2)^2}\right] \\
  & & \geq \frac{1}{2}\left[\frac{(\frac{15}{16}(x^2 +(\frac{1}{N})^2))^2 - 4(x^2 +(\frac{1}{N})^2)\frac{1}{16}(x^2 +(\frac{1}{N})^2)}{([x^2 +(\frac{1}{N})^2]^2+4(x^2 +(\frac{1}{N})^2)\frac{1}{16}(x^2 +(\frac{1}{N})^2))^2} +\frac{1}{N}\sum_{k=1}^{N} \right. \\
  & & \left( \frac{(\frac{15}{16}((x+\frac{k}{N})^2+(\frac{1}{N})^2))^2}{([(x+\frac{k}{N})^2 +(\frac{1}{N})^2]^2+4((x+\frac{k}{N})^2 +(\frac{1}{N})^2)\frac{1}{16}((x+\frac{k}{N})^2+(\frac{1}{N})^2))^2} \right. \\
   & & \left. \left. - \frac{4((x+\frac{k}{N})^2 +(\frac{1}{N})^2)\frac{1}{16}((x+\frac{k}{N})^2+(\frac{1}{N})^2)}{([(x+\frac{k}{N})^2 +(\frac{1}{N})^2]^2+4((x+\frac{k}{N})^2 +(\frac{1}{N})^2)\frac{1}{16}((x+\frac{k}{N})^2+(\frac{1}{N})^2))^2}\right) \right]
\end{eqnarray*}
 $\displaystyle = \frac{1}{2}\left[\frac{\frac{161}{256}(x^2 +(\frac{1}{N})^2)^2}{\frac{25}{16}[x^2 +(\frac{1}{N})^2]^4} \\
+ \frac{1}{N}\sum_{k=1}^{N} \frac{\frac{161}{256}((x+k/N)^2 +(\frac{1}{N})^2)^2}{\frac{25}{16}[(x+\frac{k}{N})^2 +(\frac{1}{N})^2]^4}\right] \\
= \frac{161}{800}\left[\frac{1}{[x^2 +(\frac{1}{N})^2]^2} \\
+ \frac{1}{N}\sum_{k=1}^{N} \frac{1}{[(x+\frac{k}{N})^2 +(\frac{1}{N})^2]^2}\right] $
\begin{equation} \label{e:dyvOmega+} \displaystyle  \implies \partial_{y}v_N \geq \frac{161}{800[x^2+(\frac{1}{N})^2]^2} \end{equation}

We recall that $y_{x,N}=\dfrac{(x^2+(\frac{1}{N})^2)^{5/4}}{4}$

$\displaystyle \implies \min_{y_{x,N}/2\leq|y|\leq
y_{x,N}}|v_N(x,y)|\geq \int_0^{y_{x,N}/2}
\frac{161}{800[x^2+(\frac{1}{N})^2]^2} \, dt \\
= \frac{161}{6400[x^2+(\frac{1}{N})^2]^{3/4}}$

From \eqr{e:angle}, we have $\langle
\gamma'_{x,N}(y),\gamma'_{x,N}(0) \rangle > \frac{1}{2} \cosh
v_N(x,y)$.

Integrating this gives $\langle
\gamma_{x,N}(y_{x,N})-\gamma_{x,N}(0),\gamma'_{x,N}(0) \rangle \\
> \int_{y_{x,N}/2}^{y_{x,N}} \frac{1}{2}\cosh( v_N(x,y)) \, dy \\
\geq \displaystyle \frac{1}{2} \frac{(x^2+(\frac{1}{N})^2)^{5/4}}{8}
\cosh\left(\frac{161}{6400[x^2+(\frac{1}{N})^2]^{3/4}}\right)$

Now, since $|\gamma'_{x,N}(0)|=\cosh v_N(x,0)=1$ and $\\ \displaystyle \lim_{s\rightarrow 0}s^{5/4}
\cosh\left(\frac{161}{6400s^{3/4}} \right)=\infty$, this result
and the analog for $\\ \gamma_{x,N}(-y_{x,N})$ give our result on
$\Omega_N^{+}$ , \eqr{e:boundary}, that
\begin{equation} \label{e:br1} \forall x \in (0,1/2], \, \left| F_N(x, \pm
\frac{(x^2+(\frac{1}{N})^2)^{5/4}}{4})-F_N(x,0)
\right| > r_1 \end{equation} for some $r_1 > 0$ (independent of $N$) and all $N \geq 2$.

\subsection{\textbf{Proof of \eqr{e:boundary} on $\Omega_N^{-}$}}
\vskip1cm We recall that $\Omega_N^{-} = \{(x,y) \, | \, |y| \leq
b_N, -1/2 \leq x \leq 0 \}$, where $b_N=\dfrac{1}{4N^{5/2}}$

Now on $\Omega_N^{-}$, from \eqr{dyvN}, and by using
\eqr{boundary2} and \eqr{boundary1}, we have:
\begin{eqnarray*}
  \lefteqn{\displaystyle \partial_{y}v_N = \frac{1}{2}\left[\frac{(x^2 +(\frac{1}{N})^2 - y^2)^2 - 4x^2y^2}{([x^2 +(\frac{1}{N})^2 - y^2]^2+4x^2y^2)^2} \right.}  \\
  && \left. + \frac{1}{N}\sum_{k=1}^{N} \frac{((x+\frac{k}{N})^2 +(\frac{1}{N})^2 - y^2)^2 -
4(x+\frac{k}{N})^2y^2}{([(x+\frac{k}{N})^2 +(\frac{1}{N})^2 - y^2]^2+4(x+\frac{k}{N})^2y^2)^2}\right] \\
  && \geq \frac{1}{2}\left[\frac{(\frac{15}{16}(x^2 +(\frac{1}{N})^2))^2 - 4(x^2 +(\frac{1}{N})^2)\frac{1}{16}(x^2 +(\frac{1}{N})^2)}{([x^2 +(\frac{1}{N})^2]^2+4(x^2 +(\frac{1}{N})^2)\frac{1}{16}(x^2 +(\frac{1}{N})^2))^2} +\frac{1}{N}\sum_{k=1}^{N} \right.\\
  && \left( \frac{(\frac{15}{16}((x+\frac{k}{N})^2+(\frac{1}{N})^2))^2}{([(x+\frac{k}{N})^2 +(\frac{1}{N})^2]^2+4((x+\frac{k}{N})^2 +(\frac{1}{N})^2)\frac{1}{16}((x+\frac{k}{N})^2+(\frac{1}{N})^2))^2} \right.\\
  && - \left. \left. \frac{4((x+\frac{k}{N})^2 +(\frac{1}{N})^2)\frac{1}{16}((x+\frac{k}{N})^2+(\frac{1}{N})^2)}{([(x+\frac{k}{N})^2 +(\frac{1}{N})^2]^2+4((x+\frac{k}{N})^2 +(\frac{1}{N})^2)\frac{1}{16}((x+\frac{k}{N})^2+(\frac{1}{N})^2))^2}\right) \right]
\end{eqnarray*}
$\displaystyle = \frac{1}{2}\left[\frac{\frac{161}{256}(x^2 +(\frac{1}{N})^2)^2}{\frac{25}{16}[x^2 +(\frac{1}{N})^2]^4} \\
+ \frac{1}{N}\sum_{k=1}^{N} \frac{\frac{161}{256}((x+\frac{k}{N})^2 +(\frac{1}{N})^2)^2}{\frac{25}{16}[(x+\frac{k}{N})^2 +(\frac{1}{N})^2]^4}\right] \\
= \frac{161}{800}\left[\frac{1}{[x^2 +(\frac{1}{N})^2]^2} \\
+ \frac{1}{N}\sum_{k=1}^{N} \frac{1}{[(x+\frac{k}{N})^2 +(\frac{1}{N})^2]^2}\right]
\\ \geq \frac{161}{800}\left[\frac{1}{N}\sum_{k=1}^{N} \frac{1}{[(x+\frac{k}{N})^2
+(\frac{1}{N})^2]^2}\right]$

Now $ \forall x \in [-\frac{1}{2},0], \, \forall N \geq 2, \,
\exists t_x \in \mathbb{Z}, \, 1 \leq t_x \leq N \, $ s.t. $\,
-\frac{t_x}{N} < x \leq
-\frac{t_x-1}{N}$ \\
$\implies x+\frac{t_x}{N} \leq \frac{1}{N}$ \\
Hence $\forall \, -\frac{1}{2} \leq x \leq 0, \,$ \\
$\displaystyle \partial_{y}v_N \geq \frac{161}{800}\frac{1}{N}\frac{1}{[(x+\frac{t_x}{N})^2+(\frac{1}{N})^2]^2} \geq
\frac{161}{800}\frac{1}{N}\frac{1}{[(\frac{1}{N})^2+(\frac{1}{N})^2]^2}$

\begin{equation} \implies \label{dyvNlowerbd} \partial_{y}v_N \geq \frac{161}{3200}N^{3} \end{equation}

$\displaystyle \implies \min_{y_{x,N}/2\leq|y|\leq
y_{x,N}}|v_N(x,y)|\geq \int_0^{y_{x,N}/2} \frac{161}{3200}N^{3} \,
dt \\
= \frac{161}{3200}N^{3} \frac{y_{x,N}}{2} = \frac{161}{6400}N^{3}b_N \\
=  \frac{161}{6400}N^{3}\left(\dfrac{1}{4N^{5/2}}\right) =\frac{161}{25600}N^{1/2} $\\
From \eqr{e:angle}, we have $\langle
\gamma'_{x,N}(y),\gamma'_{x,N}(0) \rangle > \frac{1}{2} \cosh
v_N(x,y)$\\
Integrating this gives \\
$\langle \gamma_{x,N}(y_{x,N})-\gamma_{x,N}(0),\gamma'_{x,N}(0) \rangle >
\int_{y_{x,N}/2}^{y_{x,N}} \frac{1}{2}\cosh( v_N(x,y)) \, dy \\
\geq \displaystyle \frac{1}{2} \frac{b_N}{2} \cosh\left(\frac{161}{25600}N^{1/2}\right)$
\begin{equation} \label{bdryinfinityproof}\implies \langle
\gamma_{x,N}(y_{x,N})-\gamma_{x,N}(0),\gamma'_{x,N}(0) \rangle > \dfrac{1}{16N^{5/2}}
\cosh\left(\frac{161}{25600}N^{1/2}\right)  \end{equation}

Now, since $|\gamma'_{x,N}(0)|=\cosh v_N(x,0)=1$ and $\\ \displaystyle \lim_{N\rightarrow
\infty}\dfrac{1}{16N^{5/2}}
\cosh\left(\frac{161}{25600}N^{1/2}\right) =\infty$, this result
and the analog for $\\ \gamma_{x,N}(-y_{x,N})$ give our result on
$\Omega_N^{-}$ , \eqr{e:boundary}, that
\begin{equation} \label{e:br2} \forall x \in [-1/2,0], \, \left|
F_N(x, \pm b_N)-F_N(x,0)
\right| > r_2 \end{equation} for some $r_2 > 0$ (independent of $N$) and all $N \geq 2$.

Hence, by choosing $r_0=\min\{r_1,r_2\}$ given by \eqr{e:br1} and \eqr{e:br2}, we have \eqr{e:boundary}.\\\\
Also by \eqr{bdryinfinityproof} we have the result \eqr{e:boundaryinfinity} that for $x \leq 0,\,|F_N(x , \pm b_N) - F_N(x,0)| \rightarrow \infty \, \text{as} \, N \rightarrow \infty$.

\end{proof}

Now we will prove the following corollary that gives us the embeddings $F_N$ that we will use in the proof of Theorem \ref{t:main1}.

\begin{Cor} \label{c:corembeddings}
Let $r_0$ be given by \eqr{e:boundary}.
\begin{enumerate}
 \item[(i)]
$F_N$ is an embedding and $F_N \left(\Omega_N \right) \subset \RR^2 \times [-1/2,1/2] \subset \RR^3$.
 \item[(ii)]
 $F_N (t,0) =
(0,0,t)$ for $|t| \leq 1/2$.
 \item[(iii)]
 $F_N \left(\Omega_N \right) \setminus \{(0,0,t) |\, |t| \leq 1/2 \} = {M}_{1,N} \cup {M}_{2,N}$
 for multi-valued graphs ${M}_{1,N} , \, {M}_{2,N}$ over the $\{x_3=0\} \setminus \{0\}$ punctured plane.
\end{enumerate}
\end{Cor}

\begin{proof}

(i) follows from \eqr{e:zequalx} and \eqr{e:graphical}.

We obtain (ii) by integrating \eqr{e:dxF} with respect to $x$,
using the fact that $v_N(x,0)$ is identically $0$, and the fact
that $F(0,0)=(0,0,0)$ because
of our choice of $z_0=0$ in \eqr{weierstrassrep}.
From (\ref{n}), $F_N$ is "vertical" (i.e. $\langle
\textbf{n},(0,0,1) \rangle = 0$) exactly when $|g_N|=1$.
But since $g_N(z)=e^{i(u_N(z)+iv_N(z))}, \, |g_N(x,y)|=1 \iff
v_N(x,y)=0$.
Now for $x>0$, by \eqr{e:dyvOmega+} and since $v_N(x,0)=0$, we have $\dd |v_N(x,y)| \geq \dfrac{161 |y|}{800[x^2+(\frac{1}{N})^2]^2}$.
Similarly, for $x \leq 0$, by \eqr{dyvNlowerbd} and since $v_N(x,0)=0$, we have $\dd |v_N(x,y)| \geq \dfrac{161 |y|}{3200}N^{3}$.
Hence, for all $x,\, \dd v_N(x,y)=0 \iff y=0$ and therefore $\langle
\textbf{n},(0,0,1) \rangle = 0 \iff y=0$.

Therefore by Corollary \ref{c:corembeddings}, (ii), the image of
$F_N$ is graphical away from the $x_3$-axis, giving us (iii).

\end{proof}

\subsection{Proof of Theorem \ref{t:main1}} $$$$

Corollary \ref{c:corembeddings} gives us a sequence of minimal embeddings
$F_N: \Omega_N \to \RR^2 \times [-1/2,1/2] \subset \RR^3$ with $F_N (t,0) = (0,0,t)$ for $|t| \leq 1/2$.

We let $M_{N} = F_{N}(\Omega_{N})$.
\subsubsection{Proof of Theorem \ref{t:main1} (a)}
\vskip1cm By using \eqr{K} with our Weierstrass data, $g (z) = \e^{i \,
(u(z)+i v(z))}$ and $\phi= dz$, we have that the curvature of $F_N$ is
given by \begin{equation} \label{curveq} \displaystyle K_N(z)=
\frac{-{|\partial_z h_N|}^2}{\cosh^4{v_N}}\end{equation}
Therefore, if $|\partial_z h_N|
\rightarrow \infty$ and for some constant $M>0, \cosh^4{v_N}<M$, then $K_N \rightarrow \infty$

Let $z \in [-1/2,0]$.
$\forall N \geq 2, \, \exists t_z \in \mathbb{Z}, \, 1 \leq t_z
\leq N \, $ s.t. $\, -\dfrac{t_z}{N} < z \leq
-\dfrac{t_z-1}{N}$
$\implies z+\dfrac{t_z}{N} \leq \dfrac{1}{N}$

$\displaystyle \implies
\partial_z h_N(z)=\frac{1}{2}\left[\frac{1}{[z^2+(\frac{1}{N})^2]^2}+\frac{1}{N}\sum_{k=1}^{N}
\frac{1}{[(z+\frac{k}{N})^2+(\frac{1}{N})^2]^2} \right] \\ \geq
\frac{1}{2N} \frac{1}{[(z+\frac{t_z}{N})^2+(\frac{1}{N})^2]^2}$
$\displaystyle \geq \frac{1}{2N} \frac{1}{[(\frac{1}{N})^2+(\frac{1}{N})^2]^2}= N^3/8$
$\displaystyle \implies \lim_{N \rightarrow \infty}{\partial_z
h_N(z)}= \infty$

Now, we also note that $\forall x, \, -1/2 \leq x \leq 1/2, \,
v_N(x,0)=0. \\ \implies \cosh^4(v_N(x,0))=1$

Hence, we have curvature blowing up at all points of the line
segment, $[-1/2,0] \subset \Omega_N $.

This gives us that $\displaystyle \forall p \in \{ (0,0,t) \,
| \, -1/2\leq t \leq 0 \}, \, \lim_{N\to \infty} |A_{M_N}|^2
(p) = \infty$.

\subsubsection{Proof of Theorem \ref{t:main1} (b) and (c)}

Theorem \ref{t:main1} (b) follows immediately from Corollary \ref{c:corembeddings}(iii).

To prove Theorem \ref{t:main1} (c) we fix $\delta>0,$ and let $B_{\delta}$ be a ${\delta}$-neighborhood of
$\\ \{ (0,0,t) \, | \, -1/2\leq t \leq 0 \}$ that is cylindrically shaped (shown in Figure \ref{deltaneighborhood}).
\begin{center}
\begin{figure}[ht]
 \includegraphics[width=0.5\textwidth]{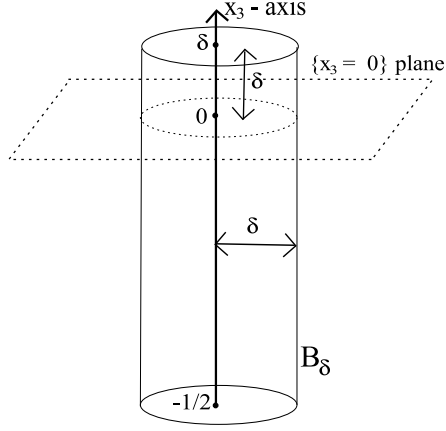}
 \caption{Diagram of $B_{\delta}$, the ${\delta}$-neighborhood of $\{ (0,0,t) \, | \, -1/2\leq t \leq 0 \}$ in Theorem \ref{t:main1}}\label{deltaneighborhood}
\end{figure}
\end{center}
Then $\forall N, \, \forall \,p = (x_1,x_2,x_3) \in M_N,$ where $x_3>\delta$ (i.e. for all points of $M_N$
that are more than a $\delta$ distance above the $x_3=0$ plane), by
(\ref{e:zequalx}), $x=x_3>\delta$, where $(x,y) \in \Omega_N^+$
such that $F_N(x,y)=(x_1,x_2,x_3)$.

Hence, since on $\displaystyle \Omega_N^+$, $ y^2 \leq
\dfrac{x^2+1/N^2}{16} \leq
\dfrac{(x+k/N)^2+1/N^2}{16}$,

$\displaystyle
|z^2+(\frac{1}{N})^2|^2=|(x+iy)^2+(1/N)^2|^2=[x^2+(1/N)^2-y^2]^2+4x^2y^2$

$\displaystyle \geq (15/16)^2[x^2+(1/N)^2]^2+4x^2y^2 \geq
(15/16)^2x^4 > (15/16)^2\delta^4$

Similarly, $\displaystyle |(z+k/N)^2+(\frac{1}{N})^2|^2=|(x+k/N+iy)^2+(1/N)^2|^2$
$\displaystyle =[(x+k/N)^2+(1/N)^2-y^2]^2+4(x+k/N)^2y^2$
$\displaystyle \geq (15/16)^2[(x+k/N)^2+(1/N)^2]^2+4(x+k/N)^2y^2 \geq (15/16)^2x^4>(15/16)^2\delta^4$

$\displaystyle \implies  \left|\partial_{z}h_N(z)\right| =
\left|\frac{1}{2}\left[\frac{1}{[z^2+(\frac{1}{N})^2]^2}+\frac{1}{N}\sum_{k=1}^{N}
\frac{1}{[(z+\frac{k}{N})^2+(\frac{1}{N})^2]^2} \right] \right| \\ <
\frac{1}{2}\left[\frac{1}{(15/16)^2\delta^4}+\frac{1}{(15/16)^2\delta^4}\right] < \frac{1}{(15/16)^2\delta^4}$

This uniform bound of $\left|\partial_{z}h_N(z)\right|$ gives the curvature bound for all points of $M_N$ that are at least a distance $\delta$ above the $x_3=0$ plane by (\ref{curveq}). Let this curvature bound be $C_{\delta}^{(1)}$.

Now, at all points $p \in M_N$ that are at least a distance $\delta$ away from the $x_3-$axis (i.e. outside a cylinder about the $x_3-$axis of radius $\delta$), Heinz's curvature estimate for graphs (11.7 in \cite{Os}) applied to components of $M_N$ over disks of radius $\delta / 2$ in the $\{x_3=0\}$ plane, which are guaranteed to be graphs over the $\{ x_3 = 0 \}$ plane by Theorem \ref{t:main1}(b), gives a uniform curvature bound, $C_{\delta}^{(2)}$.

Hence, $C_{\delta}= \max\{C_{\delta}^{(1)}, C_{\delta}^{(2)}\}$ is the uniform curvature bound in Theorem \ref{t:main1}(c).

Theorem \ref{t:main1}(d) follows from \eqr{e:boundary} and \eqr{e:boundaryinfinity}.

\subsection{Proof of Theorem \ref{t:main}}

We will need the following lemma that gives us convergence.

\begin{Lem}
\label{convergencelem}
Consider the sequence of embedded minimal disks $\{M_N\}$ given by Theorem \ref{t:main1}. Let $W=\mathcal{C} \cup H$ where $\mathcal{C}$ is the cylinder $\{ (x_1, x_2, x_3) |\, x_1^2+x_2^2 \leq r_0^2, -1/2 < x_3 < 1/2 \}$ with $r_0$ determined by Theorem \ref{t:main1}(d) and $H=\RR^2 \times [-1/2,0]$ (a horizontal block of the half space below the $\{x_3=0\}$ axis).
\begin{enumerate}
\item[(a)] $\{M_N\}$, as a sequence of minimal laminations, has a subsequence that converges to a limit lamination on compact subsets of $W$ away from $\dd \{(0,0,t) | -1/2 \leq t \leq 0\}$ in the $C^{\alpha}$ topology for any $\alpha <1$. \\
\item[(b)] This subsequence of embedded minimal disks has a further subsequence $\{M_{N_i}\}$ such that the leaves converge uniformly in the $C^k$ topology for all $k$.
\end{enumerate}
\end{Lem}

\begin{proof}
To prove Lemma \ref{convergencelem}(a) we cover compact subsets $K$ of $\displaystyle W \setminus \{(0,0,t) | -1/2 \leq t \leq 0\}$ with sufficiently small balls $B_{r_K}$ (with radius $r_K$ depending on the compact subset) such that the covering does not intersect $\{(0,0,t) |\, |t| \leq 1/2 \}$. For each $K$, we take $N$ in the sequence $\{M_N\}$ to be large enough (i.e. $N \geq N_K$ for some $N_K$ depending on $K$) such that $\partial M_N$ is outside $K$. This ensures that in each ball $B_{r_K}$, the leaves of $M_N$ in the ball have boundary contained in $\partial B_{r_K}$. Then we use the uniform curvature bound in Theorem \ref{t:main1}(c), apply Proposition B.1 in \cite{CM5} to each ball and pass to successive subsequences on each ball, to obtain a subsequence that we rename $\{M_{N}\}$ that converges to a lamination, $\mathcal{L}$, with minimal leaves on compact subsets of $W \setminus \{(0,0,t) | -1/2\leq t \leq 0\}$ in the $C^{\alpha}$ topology for any $\alpha <1$.

To prove Lemma \ref{convergencelem}(b) we consider the subsequence $\{M_{N}\}$ obtained above in Lemma \ref{convergencelem}(a) and we recall that by Theorem \ref{t:main1}(b), $M_N \setminus \{(0,0,t) |\, |t| \leq 1/2 \} = {M}_{1,N} \cup {M}_{2,N}$ for multi-valued graphs ${M}_{1,N} , \, {M}_{2,N}$ over the $\{x_3=0\} \setminus \{0\}$ punctured plane. We cover compact subsets, $K$, of $W \setminus \{(0,0,t) |\, |t| \leq 1/2 \}$ with balls $B_{r_K}$ in the same way as in the proof of Lemma \ref{convergencelem}(a) above, such that the covering does not intersect $\{(0,0,t) |\, |t| \leq 1/2 \}$. This ensures that for all $N$, the leaves in the intersection with each ball $B_{r_K}$ are graphical over a subdomain of the $\{x_3=0\} \setminus \{0\}$ punctured plane. Then since by Corollary \ref{c:corembeddings} (i), for all $N$, ${M}_{1,N} \cup {M}_{2,N}$ has bounded maximum distance from the $\{x_3=0\} \setminus \{0\}$ punctured plane , we apply Corollary 16.7 in \cite{GT} to each ball $B_{r_K}$ to obtain uniform bounds (that are functions of $r_K$ only) on the derivatives of all orders of the graphs of the leaves in $(M_N \setminus \{(0,0,t) |\, |t| \leq 1/2 \}) \cap B_{r_K}$. Then, using standard compactness results and a diagonal argument whereby we pass to successive subsequences on each ball, we obtain a subsequence $\{M_{N_i}\}$ that converges uniformly in $C^k$ for all $k$ on compact subsets of $\displaystyle W \setminus \{(0,0,t) | |t| \leq 1/2 \}$.

\end{proof}

Now, to prove Theorem \ref{t:main}, it is sufficient (by scaling) for us to show that there exists a sequence of compact embedded minimal disks $0 \in \Sigma_N \subset  B_{R} \subset \RR^3$ with $\partial \Sigma_N \subset \partial B_{R}$ for some $R>0$. Theorem \ref{t:main1} gives us a sequence of minimal embeddings $M_N= F_N(\Omega_N) \subset \RR^3$ with $F_N (t,0) = (0,0,t)$ for $|t| \leq 1/2$.

We set $R = \min \{ r_0/2 , 1/4 \}$ where $r_0$ is given by Theorem \ref{t:main1}(d) and we let $\Sigma_{N_i} = B_{R} \cap M_{N_i}$, where the sequence $N_i$ is determined by Lemma \ref{convergencelem}. We rename this sequence $\{\Sigma_N\}$.

From Theorem \ref{t:main1}(d), we see that $\partial \Sigma_{N} \subset \partial B_R$.
And from the properties satisfied by $M_{N}$ in Theorem \ref{t:main1}, Theorem \ref{t:main} (a),(b) and (c) follow immediately.

Now, we note that Corollary \ref{c:corembeddings} (iii) and the smooth convergence of the leaves in Lemma \ref{convergencelem}(b), give us that the limit minimal lamination $\mathcal{L}$ in the upper hemisphere of $B_R$ consists of a leaf in the upper hemisphere, $\Sigma$, such that $\Sigma \setminus \{x_3-\text{axis}\}= \Sigma' \cup \Sigma''$, where $\Sigma'$ and $\Sigma''$ are multi-valued graphs.

By \eqr{e:dyF} and \eqr{e:graphical}, the horizontal slices $\{ x_3 = x \} \cap \Sigma'$ and $\{ x_3 = x \} \cap \Sigma''$ are graphs in the $\{ x_3 = x \}$ plane over the line in the direction
\begin{equation}    \label{e:infv0}
    \lim_{N \to \infty} \partial_yF_N(x,0) = \lim_{N \to \infty}(\sin u_N (x,0) , - \cos u_N (x,0) , 0) \, .
\end{equation}

We note that from \eqr{dyvN}, by the Cauchy-Riemann equations, $\forall N>0,$
$ \label{dxuN} \partial_x u_N (x,0) = \partial_y v_N (x,0) $
\begin{equation}= \displaystyle \frac{1}{2}\left[\frac{1}{(x^2 +(\frac{1}{N})^2)^2}+ \frac{1}{N}\sum_{k=1}^{N} \frac{1}{((x+k/N)^2 +(\frac{1}{N})^2)^2}\right] > 0 \end{equation}
$\implies u_N(x,0)$ is monotonically increasing w.r.t. $x$, for each fixed $N$.

Therefore, for $0< t < R$ the angle turned by the line in \eqr{e:infv0} for a change in $x$ from $t$ to $2t$ is:
\begin{equation}    \label{e:infv}
    \lim_{N \to \infty} |u_N (2t,0) - u_N (t,0)| = \lim_{N \to \infty} \left| \int_t^{2t}\partial_x u_N(x,0) \, dx \right| \end{equation}
    $\displaystyle = \lim_{N \to \infty} \int_t^{2t} \frac{1}{2}\left[\frac{1}{(x^2 +(\frac{1}{N})^2)^2}+ \frac{1}{N}\sum_{k=1}^{N} \frac{1}{((x+k/N)^2 +(\frac{1}{N})^2)^2}\right] \, dx \\ \geq \lim_{N \to \infty} \int_t^{2t} \frac{1}{2}\frac{1}{(x^2 +(\frac{1}{N})^2)^2} \, dx$ $\displaystyle = \lim_{N \to \infty} \frac{N^2}{4}\left[\frac{x}{x^2+(\frac{1}{N})^2}+ N \arctan\left( Nx \right) \right]_{t}^{2t} = \frac{7}{48t^3}
    \, .$

Hence we see that, for $0< t < R$, $\{t < |x_3| < 2t \} \cap \Sigma'$ and $\{t < |x_3| < 2t \} \cap \Sigma''$ both contain an embedded $S_t$-valued graph where $S_t \geq \frac{7}{96 \pi t^3} \to
\infty$ as $t\to 0$ . It follows that $\Sigma'$ and $\Sigma''$ must both spiral into the $\{ x_3 =0 \}$ plane. In addition, a Harnack inequality in Proposition II.2.12 in \cite{CM2} gives a lower bound on the vertical separation of the sheets in both $\Sigma'$ and $\Sigma''$, for each compact subset above the $\{ x_3 =0 \}$ plane. This shows that the spiralling into the $\{ x_3 =0 \}$ plane occurs with multiplicity one.

Finally, we show that the minimal lamination $\mathcal{L}$ in the lower hemisphere of $B_R \setminus \{(0,0,t) |\, -R \leq t \leq 0 \}$ consists of a foliation by parallel planes.

We consider the sequence of embedded minimal disks $M_N$ given by Theorem \ref{t:main1} and we recall that by Corollary \ref{c:corembeddings} (iii), $M_N \setminus \{(0,0,t) |\, |t| \leq 1/2 \} = {M}_{1,N} \cup {M}_{2,N}$ for multi-valued graphs ${M}_{1,N} , \, {M}_{2,N}$ over the $\{x_3=0\} \setminus \{0\}$ punctured plane.

For arbitrary $-\frac{1}{2} \leq t \leq 0$, fixed $j=1$ or $2$ and for all $N$ we define $\Gamma_{j,N}(t)$ to be the component of $M_{j,N}$ that is contained between the planes $\{ x_3 =t \}$ and $\{ x_3 = t+\epsilon_N \}$ where $\epsilon_N$ is such that the tangent vector $\partial_yF_N(t,0)$ to $M_{N} \cap \{ x_3 =t\}$  at the $x_3$ axis (recall that this intersection is a graph in the $\{ x_3 =t\}$ plane over the line in the direction $\partial_yF_N(t,0) = (\sin u_N (t,0) , - \cos u_N (t,0) , 0)$ by \eqr{e:dyF} and \eqr{e:graphical}) turns through an angle of $4\pi$ to the direction of the tangent vector $\partial_yF_N(t+\epsilon_N,0)$ to $M_N \cap \{ x_3 =t+\epsilon_N \}$ at the $x_3$ axis, as $x$ increases from $t$ to $t+\epsilon_N$ (See Figure \ref{fourpiturn}).
This definition ensures that $\Gamma_{j,N}(t)=\{t \leq x_3 \leq t+\epsilon_N \} \cap M_{j,N}$ is a graph over the $\{x_3=0\} \setminus \{0\}$ punctured plane such that the level sets $M_{j,N} \cap \{ x_3 =x \}$ sweep out an angle of magnitude between $3\pi$ and $5\pi$ for $t\leq x \leq t+ \epsilon_N$.
\begin{center}
\begin{figure}
 \includegraphics[width=0.5\textwidth]{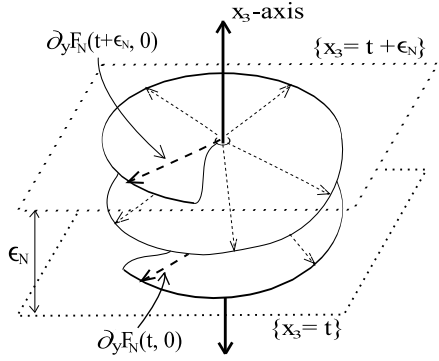}
 \caption{Diagram of $\Gamma_{j,N}(t)$ \label{fourpiturn}}
\end{figure}
\end{center}
Now we show that $\epsilon_N \to 0$ as $N \to \infty$.

For small $s$, the angle turned by the tangent vector $\partial_yF_N(x,0)$ at the $x_3$ axis for a change in $x$ from $t$ to $t+s$ is, by \eqr{dxuN}:\\
$|u_N (t+s,0) - u_N (t,0)| = \left| \int_t^{t+s}\partial_x u_N(x,0) \, dx \right|$

$\displaystyle = \int_t^{t+s} \frac{1}{2}\left[\frac{1}{(x^2 +(\frac{1}{N})^2)^2}+ \frac{1}{N}\sum_{k=1}^{N} \frac{1}{((x+k/N)^2 +(\frac{1}{N})^2)^2}\right] \, dx$

$\displaystyle \geq \frac{1}{2} \int_t^{t+s} \frac{1}{N}\sum_{k=1}^{N} \frac{1}{((x+k/N)^2 +(\frac{1}{N})^2)^2} \, dx \\ = \frac{1}{2} \frac{1}{N}\sum_{k=1}^{N} \frac{1}{((x'+k/N)^2 +(\frac{1}{N})^2)^2}((t+s) -(t))$ (for some $t \leq x' \leq t+s$ by the Mean Value Theorem)

$\dd \geq \frac{1}{2} \frac{1}{N}\frac{s}{((x'+c/N)^2 +(\frac{1}{N})^2)^2}$ (for $1 \leq c \leq N$ chosen so that $x'+c/N \leq 1/N$)

$\dd \geq \frac{1}{2} \frac{1}{N}\frac{s}{((1/N)^2 +(1/N)^2)^2} = \frac{1}{8}N^3s$

Therefore, for any $s$, $\{t \leq x_3 \leq t+s \} \cap M_{j,N}$ contains an
embedded $R_t$-valued graph where $R_t \geq \frac{1}{16 \pi}N^3 s \to \infty$ as $N \to \infty$.
This means that since for all N,  $\Gamma_{j,N}(t)=\{t \leq x_3 \leq t+\epsilon_N \} \cap M_{j,N}$ as defined above is at most $3$-valued, $\epsilon_N \leq 48\pi /N^3 \to 0$ as $N \to \infty$.

Now we have that for each $N$, $\Gamma_{j,N}(t)$ is an embedded minimal graph over the $\{x_3=0\} \setminus \{0\}$ punctured plane by \ref{t:main1}(b), the boundary of each horizontal slice of $\Gamma_{j,N}(t)$ tends to infinity by Theorem \ref{t:main1}(d), and as we have shown above, $\Gamma_{j,N}(t)=\{t \leq x_3 \leq t+\epsilon_N \} \cap M_{j,N}$ is such that $\epsilon_N \to 0$ as $N \to \infty$. Therefore, by Lemma \ref{convergencelem}(b), a subsequence $\{\Gamma_{j,N_i}(t)\}$ converges uniformly on compact subsets in the $C^k$ topology for all $k$ to an entire minimal graph minus the point $(0,0,t)$. By a standard Bernstein type theorem, this limit graph must be a plane with a removable singularity at $(0,0,t)$.
Since $-1/2\leq t \leq 0$ was arbitrary we have that the limit lamination $\mathcal{L}$ below the $\{x_3=0\}$ plane is a foliation by planes parallel to $\{x_3=0\}$ with removable singularities along the negative $x_3$-axis. And by intersecting with $B_R$, we obtain the required result that the lamination of the lower hemisphere of $B_R \setminus \{(0,0,t) |\, -R \leq t \leq 0 \}$ consists of a foliation by parallel planes, each with a removable singularity at the $x_3$-axis. The one exception is that the singularity at the origin is not removable because of the spiralling of the leaf in the upper hemisphere, $\Sigma$, into the $\{x_3=0\}$ plane.

\end{document}